\documentclass[12pt, a4paper]{article}
\usepackage[english]{babel}
\usepackage{amsmath, amsthm, amsfonts, amsbsy, amssymb, scrextend, graphicx}
\usepackage{hyperref} 
\usepackage{verbatim}
\usepackage{color}

\usepackage{tikz}
\usepackage[export]{adjustbox}
\usepackage{float}

 \newcommand{\R}{\mathbb{R}}
 \newcommand{\Z}{\mathbb{Z}}
 
\newcommand{\E}{\mathsf{E}}

\renewcommand{\P}{\mathsf{P}}
\newcommand{\Q}{\mathsf{Q}}

 \newcommand{\Px}{\mathsf{P}_{{\bf x}}}
\newcommand{\Py}{\mathsf{P}_{{\bf y}}}
 
 \newcommand{\bx}{{\bf x}}
 \newcommand{\by}{{\bf y}}
 \newcommand{\bz}{{\bf z}}

\newcommand{\eps}{\varepsilon}

 \newcommand{\be}{{\bf e}}

 \newtheorem{thm}{Theorem}[section]
 \newtheorem{lemma}{Lemma}[section]
 
 \newtheorem{rmk}{Remark}
\newtheorem{prop}{Proposition}[section]
\newtheorem{dfn}{Definition}[section]
\newtheorem{exam}{Example}[section]

\numberwithin{equation}{section}

\numberwithin{rmk}{section}

\begin{document}

\title{Localisation in a growth model with interaction.  Arbitrary graphs}

\author{
Mikhail Menshikov
\footnote{Department of Mathematical Sciences, Durham University, UK. \newline
\indent  Email: mikhail.menshikov@durham.ac.uk
}\\
{\small  Durham University}
\and 
Vadim Shcherbakov
\footnote{Department of Mathematics, Royal Holloway,  University of London, UK. \newline
\indent  Email: vadim.shcherbakov@rhul.ac.uk
}\\
{\small  Royal Holloway,  University of London}
}

\date{}

\maketitle


\begin{abstract}
{\small 
This paper concerns the long term behaviour of a  growth  model describing  
a random sequential allocation of particles on a finite graph. The 
probability of allocating a particle at a vertex is proportional to a log-linear function of 
numbers of existing  particles in a neighbourhood of  a vertex. When 
 this function depends  only on the number of particles in the vertex, the model becomes 
  a special case of  the generalised P\'olya urn model.
  In this special case all but finitely many particles are allocated at a single random vertex almost surely. 
In our model   interaction leads to the fact 
 that, with probability one, all but finitely many particles are allocated at  vertices 
of a maximal   clique. 
}
\end{abstract}



\section{The model and  main results}

Let $G=(V, E)$ be a non-oriented,  finite  connected graph with vertex set $V$ and  edge set $E$. 
We write  $v\sim u$ to denote that  vertices $v$ and $u$ are  adjacent,
and  $v\nsim u$,  if they are not.
By convention, $v\nsim v$ for all $v\in V$.  
Let $\Z_{+}$ be  the set of all non-negative integers and let $\R$ be the set of real numbers.
Given $\bx=(x_v,\, v\in V)\in\Z_{+}^{V}$ define the growth rates as 
\begin{equation}
\label{rates}
    \Gamma_{v}({\bf x}):=e^{\alpha x_v+\beta\sum_{u\sim v}x_u},\, v\in V,
 \end{equation}
where  $\alpha, \beta\in\R$ are two given constants.
 Consider  a discrete-time Markov chain  $X(n)=(X_{v}(n),\, v\in V)\in\Z_{+}^{V}$ with   the following 
 transition probabilities  
\begin{equation}
\label{trans}
\begin{split}
\P(X(n+1)=X(n)+\be_v|X(n)=\bx)&=\frac{\Gamma_{v}(\bx)}{\Gamma(\bx)},\quad \bx\in\Z_{+}^V,\quad v\in V,\\
\Gamma(\bx)&=\sum_{v\in V}\Gamma_{v}(\bx),
\end{split}
\end{equation}
where  $\be_v\in\Z_{+}^V$ is the $v$-th unit vector and $\Gamma_{v}(\bx)$ is  defined in~\eqref{rates}.

\begin{dfn}
\label{growth}
{\rm 
The Markov chain $X(n)=(X_v(n),\, v\in V)\in \Z_{+}^V$ with transition probabilities~\eqref{trans} 
is called the growth process with parameters $(\alpha, \beta)$ on the graph $G=(V, E)$.
}
\end{dfn}
The growth process $X(n)=(X_v(n),\, v\in V)$ 
describes a random sequential allocation of particles  on the  graph, where 
  $X_v(n)$ is interpreted as the  number of particles at vertex $v$  at time $n$. 
  The growth  process 
  can be regarded as a particular variant of an  interacting  urn model on a graph. The latter is a  probabilistic model 
  obtained from   an urn model by adding 
graph based interaction  (e.g.,~\cite{Benaim} and~\cite{FFK}).
  The growth  process  is motivated by cooperative sequential adsorption model (CSA).  CSA is widely used in physics 
and chemistry for modelling various adsorption processes (\cite{Evans}).
The main peculiarity of adsorption processes
is that adsorbed particles can change adsorption properties of the material. For
instance, the subsequent particles might be more likely adsorbed around the 
locations of previously adsorbed particles.
  In this paper we study the long term behaviour of
     the growth process with positive parameters $\alpha$     and $\beta$.  
Positive parameters generate strong interaction so    that existing particles   increase the growth rates in the neighbourhood of 
        their locations. This results in that, with probability one, all but finitely many particles are allocated at vertices of a maximal  clique (see Definition~\ref{D0} below).   
In a sense, the localisation effect is similar to  localisation phenomena  observed in other random processes with reinforcement (e.g.~\cite{Shapira} and~\cite{Volkov}).

The growth rates  defined in equation~\eqref{rates}  can be generalised as follows
  \begin{equation}
\label{rates11}
    \Gamma_{v}({\bf x})=e^{\alpha_vx_v+\sum_{ u\sim v}\beta_{vu}x_u},\quad v\in V,\quad\bx=(x_u,\, u\in V),  
 \end{equation}
 where $(\alpha_v,\, v\in V)$ and $(\beta_{vu},\, v,u\in V)$ are  arrays of real numbers. 
  Setting 
 $\alpha_v\equiv \alpha$, $\beta_{vu}\equiv \beta$ gives the growth process defined in  Definition~\ref{growth}. 
Originally, the growth process with parameters  $\alpha_v=\beta_{vu}\equiv \lambda\in \R$
 on a cycle  graph $G$  was studied in~\cite{SV10}.
 The limit cases of the model in~\cite{SV10} ($\lambda=\infty$ and $\lambda=-\infty$
 with convention $\infty\cdot 0=0$)  were studied in~\cite{SV10q}.
 The growth process  on a cycle graph $G$ and 
with growth rates given by~\eqref{rates11}, where $\alpha_v=\beta_{vu}=\lambda_v>0$, $v,u\in V$, 
   was studied in~\cite{CMSV}.
Note that if  $\beta=0$ in~\eqref{rates},  then the growth process is a special case of the  generalised P\'olya urn model with exponential weights
(see, e.g. \cite{Davis}).

\bigskip 

We need the following definitions from  the graph theory.  
\begin{dfn}
\label{D0}
{\rm 
Let  $G=(V, E)$ be a finite graph.
\begin{enumerate}
\item[1)]
Given a subset  of vertices $V'\subseteq V$  the corresponding induced subgraph is a  graph  $G'=(V', E')$ 
whose edge set $E'$ consists of all of the edges in $E$ that have both endpoints in $V'$.
 The induced subgraph $G'$ is also known as a subgraph induced by  vertices $v\in V'$.
\item[2)]  
A complete induced subgraph is called a clique.
A maximal clique is a clique  that is  not an induced subgraph of another clique. 
\end{enumerate}
}
\end{dfn}

Theorem~\ref{T1}  below is  the main result of the paper.

\begin{thm}
\label{T1}
Let $X(n)=(X_v(n),\, v\in V)\in \Z_{+}^V$ be a growth process with parameters 
$(\alpha, \beta)$ on  a finite connected graph $G=(V, E)$ with at least two vertices and let $0<\alpha\leq \beta$.
Then for every initial state $X(0)=\bx\in\Z_{+}^V$ with probability one there exists a random maximal
clique  with a vertex set  $U\subseteq V$ such that 
\begin{equation*}
\begin{split}
\lim_{n\to \infty} X_v(n)&=\infty\, \text{  if and only if  } v\in U,\, \text{ and}\\
\lim_{n\to \infty}\frac{X_{v}(n)}{X_{u}(n)}&=e^{C_{vu}},\,\text{ for } \,v,u\in U,
\end{split}
\end{equation*}
where
\begin{align*}
C_{vu}&=\lambda\lim_{n\to \infty}\sum_{w\in V}X_{w}(n)[{\bf 1}_{\{w\sim v, w\nsim u\}}-
{\bf 1}_{\{w\sim u, w\nsim v\}}], \text{ if }\,  0<\lambda:=\alpha=\beta,\, \text{and}\\
C_{vu}&=0,\, \text{ if }\, 0<\alpha<\beta.
\end{align*}
\end{thm}

\begin{rmk}
{\rm 
In other words,
Theorem~\ref{T1} states that, with probability one, starting from a finite 
random time moment  {\it all} subsequent particles 
are allocated at a random  maximal clique. This is what we call localisation of the growth process.
Note that quantities $C_{vu}$
are  {\it random} and   depend on the state of the   process  at the time moment, when localisation starts at the maximal clique.
}
\end{rmk}
\begin{exam}
{\rm 
In Figure~\ref{figure}  we provide an example of a connected graph, where  a growth process with parameters 
$0<\alpha\leq \beta$ can localise  in five possible ways.
The graph has  eight vertices labeled by numbers $1,2,3,4,5,6,7$ and $8$. There are five  maximal cliques induced by 
vertex sets $\{1,2\}$, $\{2,7\}$, $\{4,8\}$, $\{7,8\}$, $\{4,5,6\}$ and $\{2,3,4,5\}$ respectively. 
By Theorem~\ref{T1}, a growth process  with parameters  $0<\alpha\leq \beta$ can localise at any of these maximal  cliques 
with positive probability, and no other limiting behaviour is possible. 
}
\end{exam}

\begin{figure}[H]
  \centering
  \includegraphics[scale=0.9]{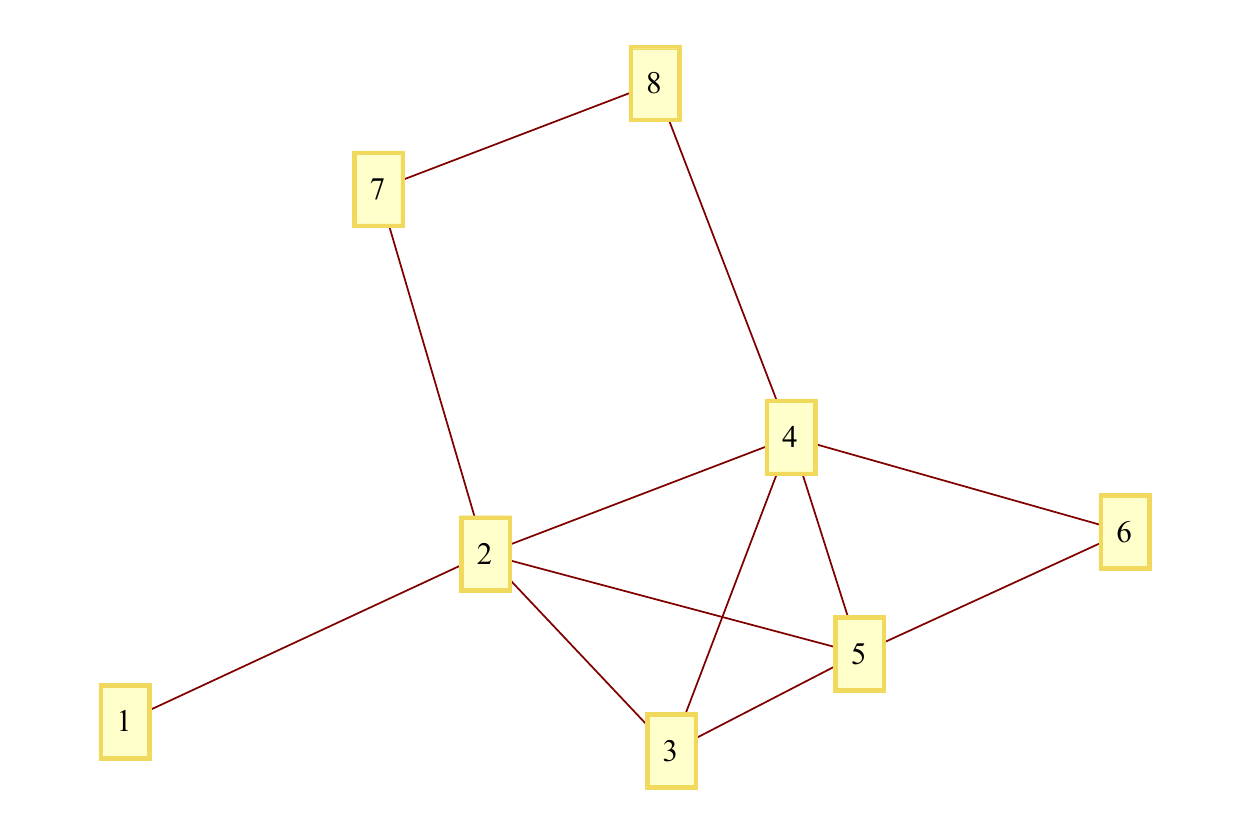}
  \vspace*{-4mm}
  \caption{\small{Graph with five cliques}}
  \label{figure}
\end{figure}

For completeness, we state and prove the following result concerning the limit behaviour of the growth process 
in the case $0<\beta<\alpha$.

\begin{thm}
\label{T2}
Let  $X(n)=(X_v(n),\, v\in V)\in \Z_{+}^V$ be a growth process with parameters 
$(\alpha, \beta)$ on a finite  connected graph $G=(V, E)$ and let $0<\beta<\alpha$. Then 
  for every  initial state $X(0)=\bx\in\Z_{+}^V$ with  probability one there exists 
  a random vertex $v$ such that 
  $$\lim_{n\to \infty} X_u(n)=\infty \text{   if and only if   } u=v.$$
 In other words, with probability one, all but a finite number of particles 
are allocated at a single random vertex.
\end{thm}

\begin{rmk}
{\rm
It is noted above, that if $\beta=0$, i.e. in the absence of interaction, our model  becomes a special case of the generalised P\'olya urn model, 
where  a particle is allocated at  a vertex $v$   with probability proportional to $e^{\alpha x_v}$, 
if the process is at  state $\bx=(x_v,\, v\in V)\in\Z_{+}^V$. In this case
 all but a finite number of particles are  allocated at a random single vertex with probability one, if $\alpha>0$. 
 Note that this particular result follows from a well known more general result for the generalised P\'olya urn model
 (\cite{Davis}).
 The attractive interaction introduced in our model 
 by a positive parameter $\beta$ leads to the fact that the growth process localises at a maximal clique rather than at a single vertex.
}
\end{rmk}
\begin{rmk}
{\rm 
In~\cite{JSV19} and~\cite{SV15} the  long term behaviour  of a continuous time Markov chain (CTMC)  $\xi(t)\in\Z_{+}^V$, where 
$V$ is  vertex set of a finite graph $G(V, E)$,   was studied.
Given state $\bx=\xi(t)\in\Z_{+}^V$ a component $\xi_v(t)$ of the Markov chain increases by one  with the rate equal to the growth  rate 
$\Gamma_v(\bx)$ defined 
in~\eqref{rates}, and a non-zero component decreases by one with the unit rate. 
Both papers~\cite{JSV19} and~\cite{SV15} were mostly concerned with classification of the long term behaviour of the Markov chain, namely, 
whether the Markov chain is recurrent or transient depending on both the  parameters $\alpha, \beta$ and the graph $G$.
The typical asymptotic behaviour of the Markov chain was studied in~\cite{SV15}  in some transient  cases. 
First of all,  it was shown in~\cite{SV15} that if both $\alpha>0$ and $\beta>0$, then, with probability one, there is a random finite time 
after   which none of the components  of CTMC $\xi(t)$  decreases.
 In other words,  with probability one,  the corresponding discrete time  Markov chain (known also as the embedded  Markov chain)  asymptotically 
 evolves  as the growth process with parameters $(\alpha, \beta)$.
Further,  if $0<\beta<\alpha$, then,   with probability one, a single component of CTMC $\xi(t)$  explodes. Theorem~\ref{T2}  above is
basically the same  result formulated in terms of the growth process.
Another result of paper~\cite{SV15} is that if $0<\alpha\leq \beta$ and the  graph $G$ is connected, has at least two vertices and  
 does not have  cliques of size more than $2$, then, with probability one,  only a pair of adjacent components of the Markov chain explodes.
Theorem~\ref{T1} in the present paper  yields the following generalisation of this result on the case of arbitrary graphs. 
Namely,   if $0<\alpha\leq \beta$, then, with probability one, only a group of CTMC $\xi(t)$   components labeled by vertices of a maximal 
clique explodes.   }
\end{rmk}

\begin{rmk}
{\rm 
Note also that in the case of a cycle graph and $\alpha=\beta>0$ localisation of the growth process at a pair of 
adjacent vertices  was 
previously shown   in~\cite[Theorem 3]{SV10} and~\cite[Theorem 1]{CMSV}.
}
\end{rmk}

Let us briefly comment on proofs of Theorems~\ref{T1} and~\ref{T2}. 
In both cases, given any initial state $X(0)=\bx$ we identify special events  that 
result in localisation of the growth process as described in the theorems.
We  show that the probability of any  event of interest is  bounded below uniformly over initial configurations.
Then  it follows from a renewal argument that almost surely 
 one of these  events eventually occurs.
Note that the same renewal argument   was used in~\cite{CMSV}.
  
In the case of  Theorem~\ref{T2}  we show by a direct computation that  given any initial state $X(0)=\bx$,
 with positive probability (depending only on the model parameters),
 all particles will be  allocated at a single vertex with the maximal growth rate.

 In the case of  Theorem~\ref{T1},  we start with detecting a maximal clique,  where  the growth process can potentially localise.
To this end, we use  a special algorithm explained  in Section~\ref{max}. Given 
any initial state $X(0)=\bx$ the algorithm outputs  a maximal 
 clique satisfying certain conditions (we call it final maximal clique, see Section~\ref{max}). 
The key step in the proof is to obtain a uniform lower bound for the probability  that  all particles 
are allocated at  vertices of a final  maximal  clique (Lemma~\ref{L1} below).
Given that all particles are allocated at vertices of a maximal 
 clique we show that  the pairwise ratios of numbers of allocated particles at the clique vertices 
converge, as claimed in Theorem~\ref{T1}.
If $\alpha=\beta$, then  convergence of the ratios 
 follows from  the strong law of large numbers for the i.i.d. case and a certain stochastic dominance argument.
If $\alpha<\beta$, then for complete graphs convergence of the ratios follows from 
a strong law of  large numbers for these graphs  (Lemma~\ref{L3}). In the case $\alpha<\beta$ and arbitrary graphs
the convergence of ratios follows from the result for complete graphs 
combined with the  stochastic dominance argument.


The rest of the paper is organised as follows. 
 In Section~\ref{prelim}, we introduce notations  and give definitions used in the proofs.
The proof of Theorem~\ref{T1} appears in Section~\ref{proofT1}, and 
Theorem~\ref{T2} is proved in Section~\ref{proofT2}.

\section{Preliminaries}
\label{prelim}

\subsection{Partition of the graph} 

Let $G=(V, E)$ be a finite connected  graph with at least two vertices. Let $G(v_1,...,v_m)$  denote a subgraph induced by vertices $v_1,...,v_m$.

\begin{dfn}
\label{D1}($D$-{\rm sets}.)
{\rm 
Let  $(v_1,...,v_{m})\subseteq V$ be an ordered subset of vertices and let subgraph 
$G(v_1,...,v_m)$ be a maximal  clique. Define the following subsets of vertices   $D_{v_1},..., D_{v_m}$  
\begin{enumerate}
\item[1)] $D_{v_1}=\{v\in V: v\nsim v_1  \text{ and } v\neq v_1\}$ and
\item[2)] $D_{v_k}=\{v\in V: v\nsim v_k,\, v\neq v_k  \text{ and }  v\sim v_1,..., v_{k-1}\}$ for   $2\leq k\leq m.$
\end{enumerate}
}
\end{dfn}

It follows from the definition of  $D$-sets that
\begin{align}
\label{Da}
\{v_1,\ldots,v_{m}\}&\cap D_{v_k}=\emptyset,\, k=1,..,m,\\
D_{v_k}&\cap D_{v_j}=\emptyset,\,v_k\neq v_j\, \text{  for }\, v_k, v_j\in\{v_1,...,v_{m}\},\label{Db}\\
V&=\{v_1,\ldots,v_{m}\}\cup D_{v_1}\cup\ldots\cup D_{v_{m}}.\label{Dc}
\end{align}

\begin{exam}
{\rm 
It should be noted that a $D$-set can  be empty. 
For instance,  let $G$ be the graph in Figure~\ref{figure}. Consider the clique with ordered set of vertices $(v_1=1, v_2=2)$,
i.e.
$G(1,2)$. Then $D_{v_1}:=D_{1}=\{3,4,5,6,7,8\}$ and $D_{v_2}:=D_2=\emptyset$. On the other hand, for the  clique 
$G(v_1, v_2):=G(2,1)$, i.e. the clique with the reverse order of vertices, we have  that $D_{v_1}=:D_2=\{6,8\}$ and $D_{v_2}:=D_1=\{3,5,4,7\}$.


}
\end{exam}

\subsection{Measure $\Q_{{\bf x}, n}$ }
In this section we introduce an auxiliary probability measure associated with the growth process. This measure naturally appears in the proof of
Lemma~\ref{L1} below and  plays an important role in the proof.

Let  $v_1,...,v_m$ be an ordered set of vertices such that the induced graph $G(v_1,...,v_m)$ is a maximal clique 
and let $D_{v_1},...,D_{v_m}$ be  the corresponding $D$-sets. 
Define
\begin{equation}
\label{Vk}
V_k:=\{v_k\}\cup D_{v_k},\, k=1,...,m.
\end{equation}

Given $i\in \{1,...,m\}$ define the following events
\begin{align}
\label{Event21}
A_{n}^{v_i}&=\{\text{at time }\,  n\, \text{  a particle is placed at site }\, v_i\},\,  n\geq 1,\\
A_{n}^{V_i}&=\{\text{at time }\,  n\, \text{  a particle is placed at site }\, v\in V_i\},\,  n\geq 1.
\label{Event22}
\end{align}
Let 
$\Px(\cdot) = \P(\cdot|X(0)=\bx)$ denote the distribution of the growth process started at $\bx\in \Z_{+}^V$.
Define  the following set of vertex sequences 
\begin{equation}
\label{Sn}
S(n)=\{(k(1),...,k(n)): k(i)\in (1,...,m),\, i=1,...,n\},\,  n\geq 1.
\end{equation}
A sequence $(k(1),...,k(n))\in S(n)$ corresponds to an event, where 
 a particle is allocated at vertex $v_{k(i)}\in (v_1,...,v_m)$ at time $i$, $i=1,...,n$. 
\begin{rmk}
{\rm 
 Note that a sequence $(v_{k(1)},...,v_{k(n)})\in S(n)$ uniquely determines  a path 
    $\bx(1),...,\bx(n)$  of length $n$ of the growth process, where 
$$\bx(j)=\bx+\sum_{i=1}^j\be_{v_{k(i)}},\, j=1,...,n.$$
}
\end{rmk}
It is easy to see   that for each  $(v_{k(1)},...,v_{k(n)})\in S(n)$
\begin{equation}
\label{aux00}
\Px\left(A_{j+1}^{V_{k(j+1)}}\Biggl|
\bigcap_{i=1}^{j}A^{v_{k(i)}}_{i}\right)=\P_{\bx+\sum_{i=1}^j\be_{v_{k(i)}}}\left(A_{1}^{V_{k(j+1)}}\right), \,\,  j=0,...,n-1.
\end{equation}
Let $\Q_{{\bf x},n}$ denote a measure on $S(n)$ defined as follows
\begin{align}
\label{aux}
\Q_{{\bf x},n}((v_{k(1)},...,v_{k(n)}))&=\Px\left(A^{V_{k(1)}}_1\right)\prod_{j=1}^{n-1}
\P_{\bx+\sum_{i=1}^j\be_{v_{k(i)}}}\left(A_{1}^{V_{k(j+1)}}\right).
\end{align}

It follows from equations~\eqref{Da}-\eqref{Vk} 
that  $V_k,\, k=1,...,m$, is a partition of the vertex set $V$ of the graph. In turn, this fact implies the following 
 proposition. 
\begin{prop}
\label{PQ} 
$\Q_{{\bf x},n}$ is a probability measure on $S(n)$, that is 
\begin{equation}
\sum_{(v_{k(1)},...,v_{k(n)})\in S(n)} \Q_{{\bf x},n}((v_{k(1)},...,v_{k(n)}))=1,
\end{equation}
where the sum is taken over all elements of  $S(n)$.
\end{prop}

\subsection{Final maximal clique}  
\label{max}

For every initial state $X(0)=\bx$ we detect 
 a maximal clique, where the growth process can potentially localise, by using an algorithm described below.
Denote for short $\Gamma_v=\Gamma_v(\bx),\, v\in V$.

\begin{itemize}
\item {\it Step 1}.
Let  $v_1$ be a vertex such that $\Gamma_{v_1}=\max(\Gamma_v: v\in V)$. 
If there are several vertices with the maximal growth rate, then choose any of these vertices arbitrary.

\item {\it Step 2}. 
Given vertex $v_1$ with the maximal growth rate, let  $v_2$ be a vertex such that  
 $ \Gamma_{v_2}=\max(\Gamma_{v}: v\sim v_1).$
If there is more than one such  vertex, then choose any of them arbitrarily. 
By construction, a subgraph $G(v_1, v_2)$ induced by vertices $v_1$ and $v_2$  is complete and $\Gamma_{v_1}\geq \Gamma_{v_2}$.
 If  $G(v_1, v_2)$ is a maximal clique,  then the algorithm terminates and outputs the maximal   clique   $G(v_1, v_2)$.
 Otherwise, the algorithm  continues.
\item {\it General step.}
Having selected  vertices $v_1,..., v_k$ such that a subgraph $G(v_1, v_2, \ldots, v_k)$ induced 
by these vertices is complete and  $\Gamma_{v_1}\geq \Gamma_{v_2}\geq\ldots\geq\Gamma_{v_k}$, proceed as follows.
If  $G(v_1, v_2, \ldots, v_k)$ is a maximal clique,
 then 
the algorithm terminates  and outputs the  maximal  clique  $G(v_1, v_2, \ldots, v_k).$ If  $G(v_1, v_2, \ldots, v_k)$ is not a maximal clique, then 
select  a vertex $v_{k+1}$ such that  
$\Gamma_{v_{k+1}}=\max\left(\Gamma_{v}:  v\sim v_j,\, j=1,\ldots, k\right).$
If  there is  more than one  such  vertex, then choose any of them arbitrary. 
In other words,  at this step of the algorithm, we select a vertex $v_{k+1}$ such that  
$v_{k+1}\sim v_j$, $j=1,...,k$, and $\Gamma_{v_1}\geq \Gamma_{v_2}\geq\ldots\geq\Gamma_{v_k}\geq \Gamma_{v_{k+1}}$.
Having selected $v_{k+1}$ repeat the general step with  complete subgraph  $G(v_1, \ldots, v_k,  v_{k+1})$.
\end{itemize}
\begin{dfn}
\label{max_clique}
{\rm 
  Given state $\bx\in\Z_{+}^V$  with growth rates $\Gamma_v=\Gamma_v(\bx),\, v\in V$, 
 a maximal clique $G(v_1,...,v_m)$ obtained by the algorithm  above  is called 
 a final  maximal clique for state $\bx$.
} 
\end{dfn}

Let $G(v_1,...,v_m)$ be a final maximal clique for state $\bx$.
Then 
\begin{align}
\label{maxclique_properties1}
&\Gamma_{v_1}=\max(\Gamma_v: v\in V); \\
& \Gamma_{v_1}\geq\ldots\geq \Gamma_{v_m};\label{maxclique_properties12}\\
&\Gamma_{v_{k+1}}=\max\left(\Gamma_{v}:  v\sim v_j,\, j=1,\ldots, k\right),\, k=1,...,m-1.
\label{maxclique_properties13}
\end{align}

\begin{exam}
{\rm 
 Let $G$ be the graph in Figure~\ref{figure}.  In this case, if the growth rates are such that  vertices  $5$ and $6$ are
  chosen at the first  and the second  step of the detection algorithm respectively, then the algorithm outputs final  maximal clique $G(5,6,4)$.  
}
\end{exam}

\begin{prop}
\label{P11}
Let subgraph $G(v_1,...,v_m)$ be  a final maximal clique for state $\bx\in \Z_{+}^V$ and let $D_{v_i},\,i=1,...,m$, be the corresponding $D$-sets. 
Let $(v_{k(1)},...,v_{k(n)})\in S(n)$ be such   that 
$r$ are particles allocated at
 vertex $v_{k(n)}$ during the time interval $[1, n-1]$.
Then
\begin{equation}
\label{P11E1}
\Px\left(A^{v_{k(n)}}_{n}\Biggl|A^{V_{k(n)}}_{n},\bigcap_{i=1}^{n-1}A^{v_{k(i)}}_{i}\right)
\geq 
\frac{1}{1+|V|e^{-\alpha r}},
\end{equation} 
where $|V|$ is the number of  vertices of the  graph $G=(V, E)$.
\end{prop}
\begin{proof}[Proof of Proposition~\ref{P11}.]
Observe that 
\begin{align}
\label{cond2}
\Px\left(A^{v_{k(n)}}_{n}\Biggl|A^{V_{k(n)}}_{n},\bigcap_{i=1}^{n-1}A^{v_{k(i)}}_{i}\right)
&=\Py\left(A^{v_{k(n)}}_{1}\Bigl|A_{1}^{V_{k(n)}}\right),
\end{align}
where 
$\by=\bx+\sum_{i=1}^{n-1}\be_{v_{k(i)}}.$
If $D_{v_{k(n)}}=\emptyset$, then the conditional probability  in~\eqref{cond2} 
is trivially equal to $1$. 
Suppose that $D_{v_{k(n)}}\neq\emptyset$. 
Recall that, by assumption,  there are 
$r$ particles at vertex $v_{k(n)}$ at time  $n-1$. Therefore, 
\begin{align*}
\Gamma_{v_{k(n)}}(\by)&=\Gamma_{v_{k(n)}}(\bx)e^{\alpha r+\beta(n-1-r)},\\
\Gamma_{v}(\by)&\leq \Gamma_{v}(\bx)e^{\beta(n-1-r)},\, \text{ for }\, v\in D_{v_{k(n)}}.
\end{align*}
Consequently, 
\begin{align*}
\P_{\by}\left(A^{v_{k(n)}}_{1}\Bigl|A_{1}^{V_{k(n)}}\right)
&\geq \frac{\Gamma_{v_{k(n)}}(\bx)e^{\alpha r+\beta(n-1-r)}}
{\Gamma_{v_{k(n)}}(\bx)e^{\alpha r+\beta(n-1-r)}+e^{\beta(n-1-r)}\sum_{v\in D_{v_{k(n)}}}\Gamma_{v}(\bx)}\\
&=\frac{1}{1+e^{-\alpha r}\sum_{v\in D_{v_{k(n)}}}\frac{\Gamma_{v}(\bx)}{\Gamma_{v_{k(n)}}(\bx)}}.
\end{align*}
By assumption, the subgraph  $(v_1,...,v_m)$ is a final maximal clique for the state $\bx$. This implies that 
 $\Gamma_{v_{k(n)}}(\bx)\geq \Gamma_v(\bx)$ for $v\in D_{v_{k(n)}}$ and, hence, 
$\sum_{v\in D_{v_{k(n)}}}\frac{\Gamma_{v}(\bx)}{\Gamma_{v_{k(n)}}(\bx)}\leq |V|.$
Finally, we   obtain that 
$$\P_{\by}\left(A^{v_{k(n)}}_{1}\Bigl|A_{1}^{V_{k(n)}}\right)\geq \frac{1}{1+|V|e^{-\alpha r}},$$
 as claimed.
 \end{proof}

\section{Proof of Theorem \ref{T1}}
\label{proofT1}

\subsection{Localisation in a final  maximal clique}
\label{ProofL1}


Define the following events. 
\begin{align}
\label{Event11}
A_{n}^{(v_1,...,v_m)}&=\{\text{at time }\,  n\, \text{  a particle is placed at  site }\, v\in(v_1,...,v_m)\},\,n\in \Z_{+},\\
A_{[1,n]}^{(v_1,...,v_m)}&=\bigcap_{k=1}^nA_{k}^{(v_1,...,v_m)},\,\, n\in \Z_{+}\cup\{\infty\}\label{Event12}.
\end{align}

\begin{lemma}
\label{L1}
 Let $X(n)=(X_v(n),\, v\in V)$ be a growth process with parameters $(\alpha, \beta)$  on a finite connected 
  graph  $G=(V, E)$  with at least  two vertices.
  Given a state $\bx\in \Z_{+}^V$ let a subgraph   $G(v_1,...,v_{m})$ be a final  maximal  
 clique for the  state $\bx$, and let $0<\alpha\leq \beta$. 
Then
 there exists  $\eps>0$ depending  only on $\alpha$ and the number of the graph vertices such that 
\begin{equation}
\label{b}
\Px\left(A^{(v_1,...,v_m)}_{[1,\infty]}\right)
\geq \eps.
\end{equation}
In other words,  all particles can be  allocated 
at vertices of a final  maximal clique with probability that is  not less than some $\eps>0$ not depending on the initial state.

\end{lemma}

\begin{proof}[Proof of Lemma~\ref{L1}]

It is easy to see that 
$$A_{[1,n]}^{(v_1,...,v_m)}=\bigcup_{(v_{k(1)},...,v_{k(n)})\in S(n)}\left( \bigcap_{i=1}^nA_{i}^{v_{k(i)}}\right),$$
where  events $A_{i}^{v_{k(i)}}$ are defined in~\eqref{Event21} , 
$S(n)$ is the set of sequences defined in~\eqref{Sn}, and the union is taken over all elements of  $S(n)$.
Therefore
\begin{equation}
\label{total}
\Px\left(A_{[1,n]}^{(v_1,...,v_m)}\right)=\sum_{(v_{k(1)},...,v_{k(n)})\in S(n)} \Px\left(\bigcap_{i=1}^n A_{i}^{v_{k(i)}}\right).
\end{equation}
Next, given  $(v_{k(1)},...,v_{k(n)})\in S(n)$ we are going to obtain a lower 
bound for the probability $ \Px\left(\bigcap_{i=1}^n A_{i}^{v_{k(i)}}\right)$.
Noting  that $A^{v_{k}}_{i}\cap A^{V_{j}}_{i}=\emptyset$ for $k\neq j$ and recalling
equation~\eqref{aux00} 
 we obtain  that 
\begin{equation}
\label{e1}
\begin{split}
\Px\left(A^{v_{k(n)}}_{n}\Biggl| \bigcap_{i=1}^{n-1}A^{v_{k(i)}}_{i}\right)&
=\Px\left(A^{v_{k(n)}}_{n}\Biggl|A^{V_{k(n)}}_{n},\bigcap_{i=1}^{n-1}A^{v_{k(i)}}_{i}\right)
\Px\left(A^{V_{k(n)}}_{n}\Biggl|\bigcap_{i=1}^{n-1}A^{v_{k(i)}}_{i}\right)\\
&=\Px\left(A^{v_{k(n)}}_{n}\Biggl|A^{V_{k(n)}}_{n},\bigcap_{i=1}^{n-1}A^{v_{k(i)}}_{i}\right)
\P_{\bx+\sum_{i=1}^{n-1}\be_{v_{k(i)}}}\left(A_{1}^{V_{k(n)}}\right).
\end{split}
\end{equation}

Suppose that $X_{v_{k(n)}}(n-1)=X_{v_{k(n)}}(0)+r$ for some $0\leq r\leq n-1$. In other words, $r$ particles are  allocated at vertex $v_{k(n)}$
 during the time interval $[1, n-1]$. Then, 
by Proposition~\ref{P11},
\begin{align}
\label{e11}
\Px\left(A^{v_{k(n)}}_{n}\Biggl|A^{V_{k(n)}}_{n},\bigcap_{i=1}^{n-1}A^{v_{k(i)}}_{i}\right)
&\geq \frac{1}{1+|V|e^{-\alpha r}}.
\end{align}
Combining~\eqref{e1} and~\eqref{e11} gives that  
\begin{equation}
\Px\left(A^{v_{k(n)}}_{n}\Biggl| \bigcap_{i=1}^{n-1}A^{v_{k(i)}}_{i}\right)
\geq \frac{1}{1+|V|e^{-\alpha r}}
\P_{\bx+\sum_{i=1}^{n-1}\be_{v_{k(i)}}}\left(A_{1}^{V_{k(n)}}\right).
\end{equation}
Consequently, 
\begin{equation}
\label{iter}
\begin{split}
\Px\left(\bigcap_{i=1}^n A_{i}^{v_{k(i)}}\right)&=\Px\left(A^{v_{k(n)}}_{n}\Biggl| \bigcap_{i=1}^{n-1}A^{v_{k(i)}}_{i}\right)
\Px\left(\bigcap_{i=1}^{n-1} A_{i}^{v_{k(i)}}\right)\\
& \geq \frac{1}{1+|V|e^{-\alpha r}}
\P_{\bx+\sum_{i=1}^{n-1}\be_{v_{k(i)}}}\left(A_{1}^{V_{k(n)}}\right)
\Px\left(\bigcap_{i=1}^{n-1} A_{i}^{v_{k(i)}}\right)
\end{split}
\end{equation}

Suppose that  $(v_{k(1)},...,v_{k(n)})$ is such that $n_i$ out of first $n$  particles are allocated  at vertex $v_i$, $i=1,..,m$, 
where $n_1,...,n_m: n_1+...+n_m=n$.
 Then, iterating equation~\eqref{iter} gives the following lower bound 
\begin{equation}
\label{e0}
\begin{split}
\Px\left(\bigcap_{i=1}^{n}A^{v_{k(i)}}_{i}\right)
&\geq \prod_{i=1}^m\left(\prod_{r=1}^{n_i-1}\frac{1}{1+|V|e^{-\alpha r}}\right)
\Q_{{\bf x},n}((v_{k(1)},...,v_{k(n)})),
\end{split}
\end{equation}
where probability $\Q_{{\bf x},n}$ is defined in~\eqref{aux}. It is easy to see   that

\begin{equation}
\label{e00}
\begin{split}
\Px\left(\bigcap_{i=1}^{n}A^{v_{k(i)}}_{i}\right)
&\geq \eps_n\Q_{{\bf x},n}((v_{k(1)},...,v_{k(n)})),
\end{split}
\end{equation}
where 
\begin{equation}
\label{eps}
\eps_n:=\left(\prod_{r=1}^{n-1}\frac{1}{1+|V|e^{-\alpha r}}\right)^m
\geq \left(\prod_{r=1}^{\infty}\frac{1}{1+|V|e^{-\alpha r}}\right)^m=:\eps>0.
\end{equation}
Therefore, for every $(v_{k(1)},...,v_{k(n)})\in S(n)$ we have that 
$$\Px\left(\bigcap_{i=1}^n A_{i}^{v_{k(i)}}\right)\geq \eps\Q_{{\bf x},n}((v_{k(1)},...,v_{k(n)})).$$
Combining the preceding display with  the fact that   
$\Q_{{\bf x},n}$ is a probability measure on $S(n)$ (Proposition~\ref{PQ}) 
 gives that 
$$\Px\left(A_{[1,n]}^{(v_1,...,v_m)}\right)\geq \eps\sum_{(v_{k(1)},...,v_{k(n)})\in S(n)}
\Q_{{\bf x},n}((v_{k(1)},...,v_{k(n)}))=\eps.$$
Consequently, 
$\Px\left(A_{[1,\infty]}^{(v_1,...,v_m)}\right)\geq \eps,$
where  $\eps>0$ (defined in~\eqref{eps})  does not depend on $\bx$. The lemma is proved.
\end{proof}

\subsection{Eventual localisation}

Let us  show that, with probability one, the growth process eventually localises at a  random maximal clique, as claimed.
To this end,  we use the renewal argument from the proof of~\cite[Theorem 1]{CMSV}.
Given  an arbitrary initial state $X(0)=(X_v(0),\, v\in V)\in \Z_{+}^V$ define 
the following sequence of random times $T_k,\, k\geq 0$. Set $T_0=0$. Suppose that time moments $T_1,...,T_k$ are defined. 
Then, given a process state $X(T_k)$ at time $t=T_k$ let $G_k$ be a final  maximal clique corresponding to state $X(T_k)$.
Define  $T_{k+1}$ as the  first time moment when a particle is allocated in a vertex not belonging  
to $G_k$.
By Lemma~\ref{L1}  $\P(T_{k+1}<\infty|X(T_k))\leq 1-\eps$ for some $\eps>0$.
This yields  that  with probability one
only a finite number of events $\{T_k<\infty\}$ occur. In other words, 
with probability one,
eventually the growth process localises at a random  maximal clique, as claimed.

\subsection{Convergence of ratios $X_v(n)/X_u(n)$}

Next we are going to show that if  all particles are allocated  at vertices of a maximal clique, then 
pairwise ratios  $X_v(n)/X_u(n)$, where $v,u$ are any two  vertices of the maximal clique, must converge,
 as claimed in Theorem~\ref{T1}. There are two cases to consider.

\subsubsection{Case: $\alpha=\beta$}
Let $\lambda:=\alpha=\beta$.
Given state $\bx=(x_v,\, v\in V)\in\Z_{+}^V$ let an induced 
subgraph $G(v_1,..., v_m)$ be a final maximal clique for state $\bx$.
Define 
\begin{equation}
\label{pi}
p_i:=\frac{\Gamma_{v_i}(\bx)}{\sum_{j=1}^m\Gamma_{v_j}(\bx)},\quad i=1,...,m.
\end{equation}
Given $\delta>0$ define 
the following subset of trajectories of the growth process
\begin{equation}
\label{Bd}
B_{\delta}=\left\{\sum_{i=1}^m|X_{v_i}(n)-p_in|\geq\delta n \, \text{ for infinitely many } n\right\}
\end{equation}
 and let   $B^c_{\delta}$ be  the complement of $B_{\delta}$. 
Then 
\begin{equation}
\label{Bdelta}
\Px\left(A_{[1,\infty]}^{(v_1,...,v_m)}\right)=\Px\left(B^c_{\delta}\bigcap A_{[1,\infty]}^{(v_1,...,v_m)}\right)+
\Px\left(B_{\delta}\bigcap A_{[1,\infty]}^{(v_1,...,v_m)}\right).
\end{equation}
\begin{prop}
\label{Pdelta}
For every $\delta>0$ and $\bx\in\Z_{+}^V$
$$\Px\left(B_{\delta}\bigcap A_{[1,\infty]}^{(v_1,...,v_m)}\right)=0.$$
 \end{prop}
 \begin{proof}[Proof of Proposition~\ref{Pdelta}]
Let $(v_{k(1)},...,v_{k(n)})\in S(n)$. 
Observe that the assumption $\alpha=\beta$ implies the following equation
$$
\P_{\bx}\left(A_{n}^{v_{k(n)}}\Biggl|
A_{n}^{(v_1,...v_m)},  \bigcap_{i=1}^{n-1}A_i^{v_{k(i)}}\right)=
p_{k(n)},
$$
 where probabilities  $p_i, i=1,...,m,$ are defined in~\eqref{pi}.
Therefore,
\begin{equation}
\label{aa10}
\begin{split}
\P_{\bx}\left(A_{n}^{v_{k(n)}}\Biggl| \bigcap_{i=1}^{n-1}A_i^{v_{k(i)}}\right)&=\P_{\bx}\left(A_{n}^{v_{k(n)}}\Biggl|
A_{n}^{(v_1,...v_m)},  \bigcap_{i=1}^{n-1}A_i^{v_{k(i)}}\right)\P_{\bx}\left(A_{n}^{(v_1,...v_m)}\Biggl|
 \bigcap_{i=1}^{n-1}A_i^{v_{k(i)}}\right)\\
&\leq\P_{\bx}\left(A_{n}^{v_{k(n)}}\Biggl|
A_{n}^{(v_1,...v_m)},  \bigcap_{i=1}^{n-1}A_i^{v_{k(i)}}\right) = p_{k(n)},
\end{split}
\end{equation} 
and, hence, 
\begin{equation}
\label{iter1}
\begin{split}
\Px\left(\bigcap_{i=1}^n A_{i}^{v_{k(i)}}\right)&=\Px\left(A^{v_{k(n)}}_{n}\Biggl| \bigcap_{i=1}^{n-1}A^{v_{k(i)}}_{i}\right)
\Px\left(\bigcap_{i=1}^{n-1} A_{i}^{v_{k(i)}}\right) \leq p_{k(n)}
\Px\left(\bigcap_{i=1}^{n-1} A_{i}^{v_{k(i)}}\right).
\end{split}
\end{equation}
Let   $(v_{k(1)},...,v_{k(n)})$ be such that 
$$\sum_{j=1}^n\be_{v_{k(j)}}=\sum_{k=1}^mn_k\be_{v_k},\, \text{ where }\, \sum_{k=1}^mn_k=n,$$
 i.e., $n_i$ out of first $n$ particles are allocated at vertex $v_i$. Then, 
 iterating equation~\eqref{iter1} gives the following upper bound for the  probability of a fixed path of length $n$ of the growth process 
\begin{equation}
\label{aa2}
\Px\left(\bigcap_{j=1}^nA_j^{v_{k(j)}}\right)\leq p_{1}^{n_1}\cdots p_{m}^{n_m}.
\end{equation}
Consider a random process $Y(n)=(Y_1(n),...,Y_m(n))$ describing results of independent 
trials, where  in each trial a particle is  allocated in one of $m$ boxes labeled by $i=1,...,m$ with respective 
 probabilities  $p_i$, $i=1,...,m$,  and  $Y_i(n)$  is the number of particles in box $i$ after $n$ trials. 
Let $\widetilde \P$ denote distribution of  this process.
  It is easy to see that 
 the right hand side of equation~\eqref{aa2} is equal to probability  $\widetilde \P(Y_i(n)=r_i,\, i=1,...,m)$, computed 
 given that the boxes are initially empty.
 Define 
  $$\widetilde B_{\delta}=\left\{ \sum_{i=1}^m|Y_{i}(n)-p_in|\geq\delta n \, \text{ for infinitely many } n\right\}.$$
Equation~\eqref{aa2} implies that 
$\Px\left(B_{\delta}\bigcap  A_{[1,\infty]}^{(v_1,...,v_m)}\right)\leq \widetilde\P\left(\widetilde B_{\delta}\right)$.
By the strong law of large numbers for the i.i.d. case we have that $\widetilde\P\left(\widetilde B_{\delta}\right)=0$,
and, hence,  
$ \Px\left(B_{\delta}\bigcap  A_{[1,\infty]}^{(v_1,...,v_m)}\right)=0$, as claimed.  
\end{proof}

It follows from Proposition~\ref{Pdelta} and equation~\eqref{Bdelta} that  
$$\Px\left(\frac{X_{v_i}(n)}{X_{v_j}(n)}\to \frac{p_i}{p_j}, \, \text{ as }\, n\to \infty
\Biggl|A_{[1,\infty]}^{(v_1,...,v_m)}\right)=1,$$
for  $i,j=1,...,m$. 
Finally, a direct computation gives  that 
$\frac{p_i}{p_j}=\frac{\Gamma_{v_i}(\bx)}{\Gamma_{v_j}(\bx)}=e^{C_{v_iv_j}},$
where 
$$C_{v_iv_j} =\lambda\lim_{n\to \infty}\sum_{w\in V}X_{w}(n)[{\bf 1}_{\{w\sim v_i, w\nsim v_j\}}-
{\bf 1}_{\{w\nsim v_i, w\sim v_j\}}],\, \text{ for }\, i,j=1,...,m.$$
The proof of  Theorem~\ref{T1} in the case $\alpha=\beta$ is now finished.

\subsubsection{Case: $\alpha<\beta$}

We start with an auxiliary statement   (Lemma~\ref{MC_Z}) that might be  of interest on its own right.

\begin{lemma}
\label{MC_Z}
Let $X(n)=(X_1(n),...,X_m(n))$ be a growth process with parameters $0<\alpha<\beta$ 
 on  a complete graph  with $m\geq 2$ vertices labeled by  $1, ...,m$, and let 
$Z_i(n)=X_i(n)-X_m(n),\, i=1,..., m-1$.  Then  
$Z(n):=(Z_1(n),..., Z_{m-1}(n))\in \Z^{m-1}$ is  an irreducible  positive recurrent Markov chain.
\end{lemma}
\begin{proof}[Proof of Lemma~\ref{MC_Z}]
Let $X(0)=\bx=(x_1,...,x_m)\in\Z_{+}^m$. 
For short, denote  $\Gamma_{i}=\Gamma_{i}(\bx),\, i=1,...,m$,  and   $-\lambda=\alpha-\beta<0$.
Note  that if  $\by=(x_1+r_1, \ldots, x_m+r_m)\in\Z_{+}^m$, where 
$\sum_{i=1}^mr_i=n$, then 
 $$\Gamma_{i}(\by)=\Gamma_{i}e^{\alpha r_i+\beta(n-r_i)}=\Gamma_{i}e^{-\lambda r_i}e^{\beta n},\quad i=1,...,m.
$$
Therefore
\begin{equation}
\label{tr_prob}
\begin{split}
\P(Z(n+1)=Z(n)+\be_i|Z(n)=\bz)&=\frac{\Gamma_{i}e^{-\lambda z_i}}
{\Gamma_{m}+\sum_{i=1}^{m-1}\Gamma_{i}e^{-\lambda z_i}},\quad i=1,...,m-1,\\
\P(Z(n+1)=Z(n)-\be
|Z(n)=\bz)&=\frac{\Gamma_m}{\Gamma_{m}+\sum_{i=1}^{m-1}\Gamma_{i}e^{-\lambda z_i}},
\end{split}
\end{equation}
for all   $\bz=(z_1,...,z_{m-1})\in \Z^{m-1}$,  
where $\be_i$ is now  the $i$-th unit vector in $\Z^{m-1}$, and $\be=\be_1+\cdots+\be_{m-1}\in\Z^{m-1}$.
Thus,  $Z(n)$ is a Markov chain with  transition probabilities given by~\eqref{tr_prob}.
It is easy to see that this Markov chain is irreducible.
Further, define the following function 
\begin{equation}
\label{f-function}
f(\bz)=\sum\limits_{i=1}^{m-1}z_i^2,\quad  \bz=(z_1,...,z_{m-1})\in \Z^{m-1},
\end{equation}
and show that given $\eps>0$ 
\begin{equation}
\label{posit}
\begin{split}
\E(f(Z(n+1))&-f(Z(n))|Z(n)=\bz)\leq -\eps,\\
\text{for }\, \bz&=(z_1,...,z_{m-1})\in\Z_{+}^{m-1}: |z_1|+...+|z_{m-1}|\geq C,
\end{split}
\end{equation}
provided that  $C=C(\eps)>0$ is sufficiently large.
Indeed, fix $\eps>0$. A direct computation gives that 
\begin{equation}
\label{bb1}
\E(f(Z(n+1))-f(Z(n))|Z(n)=\bz)+\eps=\frac{\eps+\sum_{i=1}^{m-1} h_i(z_i, \eps)}{W(\bz)},
\end{equation}
where 
$h_i(z,\eps):=(2z+1+\eps)a_ie^{-\lambda z}-2z+1$  for $z\in \R$, 
$a_{i}=\frac{\Gamma_{i}}{\Gamma_{m}},\, i=1,...,m-1$, and 
$W(\bz)=1+\sum_{i=1}^{m-1}a_ie^{-\lambda z_i}$.
It is easy to see that for each $i=1,...,m-1$, there exists $C_i>0$ such that $h_i(z,\eps)\leq-|z|$ for $|z|>C_i$. Define
$$H(\eps):=\max\limits_{i=1,...,m-1}\sup_{-\infty<z<\infty}(h_i(z,\eps)+|z|).$$
Note that  $H(\eps)>0$, as $h_i(0,\eps)=a_i+1+\eps>0$, $i=1,...,m-1$. 
It follows from the definition  of $H$ that 
$$\sum_{i=1}^{m-1} h_i(z_i, \eps)=\sum_{i=1}^{m-1} (h_i(z_i,\eps)+|z_i|)-|z_i|\leq (m-1)H(\eps)-\sum_{i=1}^{m-1}|z_i|.$$
Combining the preceding equation 
 with equation~\eqref{bb1} gives equation~\eqref{posit}, where 
 $C=\eps+(m-1)H(\eps)$. 
Thus, positive recurrence of Markov chain $Z(n)$ follows  from  the Foster criterion for positive recurrence of a Markov chain  (e.g. 
\cite[Theorem 2.6.4]{MPW})  with 
the Lyapunov function $f$.
\end{proof}

\begin{rmk}
{\rm 
Note that Lemma~\ref{MC_Z} is  reminiscent of~\cite[Theorem 1, Part (1)]{SV10}. Moreover, to show positive recurrence 
of the Markov chain $Z(n)$ we use the criterion for positive recurrence with  the same 
 Lyapunov function~\eqref{f-function} as  in the proof  of positive recurrence of a  similar Markov 
 chain in~\cite[Theorem 1, Part (1)]{SV10}.
}
\end{rmk}

The next step of the proof is to show  the convergence of the ratios in the case of a complete graph.
This is the subject of the following lemma.

\begin{lemma}[The strong law of large numbers for a growth process on a complete graph]
\label{L3}
Let $X(n)=(X_1(n),...,X_m(n))$ be a growth process with parameters $0<\alpha<\beta$ 
 on  a  complete graph  with $m\geq 1$ vertices labeled by  $1, ...,m$.
For every  initial state $X(0)=\bx\in\Z_{+}^m$ and  $\delta>0$ 
 with probability one 
 $$\sum_{i=1}^m\left|X_i(n)-\frac{n}{m}\right|>n\delta\, \text{ for finitely many }\, n.$$
In other words,  with probability one 
$\lim_{n\to \infty}\frac{X_{i}(n)}{n}=\frac{1}{m}$, $i=1,...,m$.
\end{lemma}
\begin{proof}[Proof of Lemma~\ref{L3}]
Note that 
 if $\sum_{i=1}^m|X_i(n)-n/m|>n\delta$, then $\sum_{i=1}^{m-1}|Z_i(n)|>n\delta/m^2$, where 
 $Z(n)$ is  the   Markov chain  defined in Lemma~\ref{MC_Z}.
Therefore, to prove the lemma it suffices to show that,  given  $\delta'>0$  with probability one, only a finite number of events 
 $\sum_{i=1}^{m-1}|Z_i(n)|>n\delta'$ occurs.

Let $\sigma_0=0$ and let $\sigma_k=\inf\left(n>\sigma_{k-1}: Z(n)={\bf 0}\right)$  for $k\geq 1$.
In other words, $\sigma_k$ is the $k$-th return time to the origin for the Markov chain $Z(n)$.
Define the following events 
\begin{equation}
\label{W}
W_{k,\delta'}:=\left\{\max_{n\in(\sigma_k, \sigma_{k+1})}\sum_{i=1}^{m-1}|Z_i(n)|>n\delta'\right\},\, k\geq 1.
\end{equation}
Note  that  $\sum_{i=1}^{m-1}|Z_i(n)|$ can increase  at most by $(m-1)$  at each time step, and, besides,  $\sigma_k\geq k$. 
This yields that  
\begin{equation}
\label{W-sigma}
W_{k,\delta'}\subseteq \left\{\sigma_{k+1}-\sigma_k\geq \frac{k\delta'}{m-1}\right\}.
\end{equation} 
Assume,  without loss of generality, that $Z(0)={\bf 0}$. 
Then random variables $\sigma_{k}-\sigma_{k-1},\, k\geq 1,$
are  identically distributed with the same distribution as  the first return $\sigma_1$.
It follows from  Lemma~\ref{MC_Z} that  $\E(\sigma_1)<\infty$.
Therefore,
\begin{align*}
\sum_{k=1}^{\infty}\P(\sigma_k-\sigma_{k-1}\geq k\delta'/(m-1)|Z(0)={\bf 0})
&= 
\sum_{k=1}^{\infty}\P(\sigma_1\geq k\delta'/(m-1)|Z(0)={\bf 0})\\
&\leq C\E(\sigma_1|Z(0)={\bf 0})<\infty,
\end{align*}
and, hence, by the Borel-Cantelli lemma, with probability one, only a finite number of events
$\{\sigma_k-\sigma_{k-1}\geq k\delta'/(m-1)\}$, $k\geq 1$, occur. Recalling equation~\eqref{W-sigma} gives that, 
with probability one, only a finite number of events $W_{k,\delta'}$ occur. Consequently, with probability one,
 $\sum_{i=1}^{m-1}|Z_i(n)|>n\delta'$ only for finitely many $n$, and the  lemma is proved.
 \end{proof}

Finally, we are going to show the 
convergence of the ratios  for the growth process  with parameters  $0<\alpha<\beta$ on  an arbitrary connected graph $G(V, E)$.
 Let $(v_1,...,v_{m})\subseteq V$ be vertices of a clique.
 Fix  $(v_{k(1)},...,v_{k(n)})\in S(n)$. A direct computation gives 
the following analogue  of  bound~\eqref{aa10}  
 \begin{equation}
\label{aa30}
\begin{split}
\Px\left(A_{1}^{v_{k(1)}}\right)&=\Px\left(A_{1}^{v_{k(1)}}\Bigl|
A_{1}^{(v_1,...v_m)}\right)\Px\left(A_{1}^{(v_1,...v_m)}\right)
\\
&\leq \Px\left(A_{1}^{v_{k(1)}}\Bigl|
A_{1}^{(v_1,...v_m)}\right)=
\frac{\Gamma_{v_{k(1)}}}{\sum_{k=1}^{m}\Gamma_{v_k}},
\end{split}
\end{equation} 
where, as before,  we denoted
$\Gamma_{v_k}=\Gamma_{v_k}(\bx)$, $k=1,...,m$. 
Similarly, we have 
 for every $j=2,\ldots, n$ that 
 \begin{equation}
\label{aa3}
\begin{split}
\Px\left(A_{j}^{v_{k(j)}}\Biggl|\bigcap_{i=1}^{j-1} A_{i}^{v_{k(i)}}\right)&=\Px\left(A_{j}^{v_{k(j)}}\Biggl|
A_{j}^{(v_1,...v_m)},\bigcap_{i=1}^{j-1} A_{i}^{v_{k(i)}}\right)\Px\left(A_{j}^{(v_1,...v_m)}\Biggl|
\bigcap_{i=1}^{j-1} A_{i}^{v_{k(i)}}\right)\\
&\leq \Px\left(A_{j}^{v_{k(j)}}\Biggl|
A_{j}^{(v_1,...v_m)},\bigcap_{i=1}^{j-1} A_{i}^{v_{k(i)}}\right)=
\frac{\Gamma_{v_{k(j)}}e^{-\lambda r_{k(j),j-1}}}{\sum_{k=1}^{m}\Gamma_{v_k}e^{-\lambda r_{k,j-1}}},
\end{split}
\end{equation} 
where   $\lambda=-(\alpha-\beta)$  and $r_{k,j-1}$, $k=1,...,m,$ are such that 
\begin{align*}
\sum_{i=1}^{j-1}\be_{v_{k(i)}}&=\sum_{k=1}^mr_{k,j-1}\be_{v_k}\, \text{ for }\, j\geq 2\, \text{ and }\,
 r_{k,0}=0.
\end{align*}
In other words, $r_{k,j-1}$ is the number of particles at vertex $k$  at time $j-1$. 
Then, it follows from equations~\eqref{aa30} and~\eqref{aa3}  that 
\begin{equation}
\label{aa4}
\Px\left(\bigcap_{i=1}^nA_i^{v_{k(i)}}\right)\leq 
\prod_{i=1}^n\frac{\Gamma_{v_{k(i)}}e^{-\lambda r_{k(i),i-1}}}{\sum_{k=1}^{m}\Gamma_{v_k}e^{-\lambda r_{k,i-1}}}.
\end{equation}
Consider a growth process $\widetilde X(n)$ with parameters $(\alpha, \beta)$  on the complete graph with vertices $1,...,m$,
whose
growth rates are computed as follows 
\begin{equation}
\label{rates_mod}
\Gamma_{i}(\widetilde \bx)=\Gamma_{v_i}e^{\alpha \tilde x_i+\beta\sum_{j\neq i}\tilde x_j},\, \widetilde\bx=(\tilde x_1,...,\tilde x_m)\in\Z_{+}^m,\end{equation}
where,   in contrast to growth rates~\eqref{rates},  additional coefficients   $\Gamma_{v_i}$ appear.
Assume that $\widetilde X(0)={\bf 0}$. Then, 
it is easy to see that 
 the right-hand side of equation~\eqref{aa4} is the probability of a trajectory  of length $n$ of the growth process $\widetilde X(n)$
 corresponding to  the sequence  $(v_{k(1)}, ..., v_{k(n)})\in S(n)$ as follows. This is a trajectory such that 
 a particle is allocated at vertex $k(i)\in (1,...,m)$ at time $i=1,..,n$.
Further,  given $\delta>0$ the following analogue of equation~\eqref{Bdelta}
holds 
\begin{equation}
\label{Bdelta1}
\Px\left(A_{[1,\infty]}^{(v_1,...,v_m)}\right)=\Px\left(B^{c}_{\delta}\bigcap A_{[1,\infty]}^{(v_1,...,v_m)}\right)+
\Px\left(B_{\delta}\bigcap  A_{[1,\infty]}^{(v_1,...,v_m)}\right),
\end{equation}
where  now 
$$B_{\delta}=\left\{\sum_{i=1}^m\left|X_{v_i}(n)-\frac{n}{m}\right|\geq\delta n \, \text{ for infinitely many } n\right\}$$
 and $B^c_{\delta}$ is, as before,   the complement of $B_{\delta}$. 
It follows from equation~\eqref{aa4} that
$$\Px\left(B_{\delta}\bigcap A_{[1,\infty]}^{(v_1,...,v_m)}\right)\leq \widetilde\P\left(\widetilde B_{\delta}\right),$$
where $\widetilde \P$ is the distribution of the growth process $\widetilde X(n)$ on the complete graph with $m$ vertices 
 (with growth rates~\eqref{rates_mod})
starting at $\widetilde X(0)={\bf 0}\in \Z_{+}^m$
and 
$$\widetilde B_{\delta}=\left\{\sum_{i=1}^m\left|\widetilde X_{v_i}(n)-\frac{n}{m}\right|\geq\delta n \, \text{ for infinitely many } n\right\}.$$
Note that both Lemma~\ref{MC_Z} and Lemma~\ref{L3} remain  true  for this  growth process 
(the proofs can be repeated  verbatim). Therefore, $\P\left(\widetilde B_{\delta}\right)=0$, and, hence, 
$\Px\left(B_{\delta}\bigcap  A_{[1,\infty]}^{(v_1,...,v_m)}\right)=0$.
This yields that 
$$\Px\left(\frac{X_{v_i}(n)}{X_{v_j}(n)}\to 1, \, \text{ as }\, n\to \infty
\Biggl|A_{[1,\infty]}^{(v_1,...,v_m)}\right)=1,$$
for  $i,j=1,...,m$,  as claimed. 

The proof of  Theorem~\ref{T1} in the case $0<\alpha<\beta$ is finished.

\section{Proof of  Theorem \ref{T2}} 
\label{proofT2}

Start with the following proposition.
\begin{prop}
\label{L4}
Let $X(n)=(X_v(n),\, v\in V)$ be a growth process  with parameters $(\alpha, \beta)$ 
on a finite connected graph $G=(V, E)$ and let $0<\beta<\alpha$.
Given state $\bx\in\Z_{+}^V$ with growth rates  $\Gamma_v(\bx),\, v\in V$, suppose that  $\Gamma_u(\bx)=\max(\Gamma_v(\bx): v\in V)$.
Then 
$\Px\left(A_{[1,\infty]}^{u}\right)\geq \eps
$
for some $\eps>0$ that depends only on $\alpha, \beta$ and $|V|$.
 In other words, with positive probability, 
all subsequent particles will be allocated at a vertex with the maximal growth rate.
\end{prop}
\begin{proof}[Proof of Proposition~\ref{L4}.]
The proof of the lemma is similar to the proof of Lemma 1 in~\cite{CMSV}. We provide the details  for the sake of completeness.
Note that 
$\Gamma_u(\bx+n\be_u)=\Gamma_u(\bx)e^{\alpha n}$ and 
$\Gamma_v(\bx+n\be_u)\leq \Gamma_v(\bx)e^{\beta n}$, $v\neq u$. Therefore,  using that 
$\Gamma_u(\bx)\geq \Gamma_v(\bx)$ for $v\neq u$ we obtain that 
$$\sum\limits_{v\neq u}\frac{\Gamma_v(\bx+n\be_u)}{\Gamma_u(\bx+n\be_u)}\leq |V|e^{-(\alpha-\beta)n},$$
which, in turn,  gives  that
\begin{align*}
\Px\left(A_{[1,\infty]}^{u}\right)&=\prod\limits_{n=0}^{\infty}
\frac{\Gamma_u(\bx)e^{\alpha n}}{\Gamma_u(\bx)e^{\alpha n}+\sum_{v\neq u}\Gamma_v(\bx+n\be_u)}
\geq  \prod\limits_{n=0}^{\infty}
\frac{1}{1+|V|e^{-(\alpha-\beta)n}} =:\eps>0,
\end{align*} 
as claimed.
\end{proof}
The proof of Theorem~\ref{T2} can be finished  by using the  renewal argument similarly to the
 proof of Theorem~\ref{T1}. We omit the details.

\subsection*{Acknowledgements} 

We thank Stanislav Volkov for useful comments.

\end{document}